\documentclass[12pt]{amsart}
\usepackage{fullpage,amssymb}
\newtheorem{theorem}{Theorem}[section]
\newtheorem{lemma}[theorem]{Lemma}
\newtheorem{corollary}[theorem]{Corollary}
\newtheorem{proposition}[theorem]{Proposition}
\newtheorem*{notation*}{Notation}
\newtheorem*{p*}{Proposition~\ref{h.s.o.p}}
\theoremstyle{definition}
\newtheorem{definition}[theorem]{Definition}

\newtheorem{problem}[theorem]{Problem}
\newtheorem{remark}[theorem]{Remark}

\newcommand{\R}{{\mathbb{R}}}
\newcommand{\C}{{\mathbb C}}

\newcommand{\vol}{\operatorname{vol}}

\newcommand{\Sm}{\operatorname{S}}

\usepackage{multirow}
\usepackage{amsmath}
\usepackage{comment}
\usepackage{tikz}
\usepackage{mathtools}
\usepackage{arydshln}
\usepackage{relsize}

\title{Highly entangled tensors}
\author{Harm Derksen and Visu Makam}
\thanks{This material is based upon work supported by the National Science Foundation under Grant No. DMS-1601229 and DMS-1638352}
\begin{document}

\maketitle

\begin{abstract}
A geometric measure for the entanglement of a unit length tensor $T \in (\C^n)^{\otimes k}$ is given by $- 2 \log_2 ||T||_\sigma$, where $||.||_\sigma$ denotes the spectral norm. A simple induction gives an upper bound of $(k-1) \log_2(n)$ for the entanglement. We show the existence of tensors with entanglement larger than $k \log_2(n) - \log_2(k) - o(\log_2(k))$. Friedland and Kemp have similar results in the case of symmetric tensors. Our techniques give improvements in this case.

\end{abstract}

\tableofcontents

\section{Introduction}
We consider  the tensor product $V = V_1 \otimes V_2 \otimes \dots  \otimes V_k$ of finite dimensional Hilbert
spaces $V_1,V_2,\dots,V_k$. 
A tensor of the form $T = v_1 \otimes v_2 \otimes \dots \otimes v_k \in V$ is called a {\em pure} or {\em simple} tensor. A general tensor $T \in V$ is far from pure. Each vector space $V_i$ is equipped with positive definite hermitian form $\langle \cdot,\cdot\rangle$. 
The space $V$ is again a Hilbert space, and its positive definite hermitian form $\langle\cdot,\cdot\rangle$ has  the property $\langle v_1\otimes v_2\otimes \cdots\otimes v_k,w_1\otimes w_2\otimes\cdots\otimes w_k\rangle=\prod_{i=1}^k\langle v_i,w_i\rangle$. The euclidean norm on a Hilbert space is defined by $\|v\|=\sqrt{\langle v,v\rangle}$. We call a tensor $T\in V$ a {\em unit tensor} if $\|T\|=1$.

In quantum physics, a quantum state is a unit tensor $T\in V$. A quantum state is called a {\em pure} or {\em unentangled} state if $T$ is a pure tensor.
If $T$ is not pure, then the quantum state is called a {\em mixed} or {\em entangled} state. In the theory of computation, there seems to be a significant upgrade in speed offered by quantum algorithms in comparison with classical ones. A large part of this can be attributed to entanglement. 

Various measures for entanglement have been studied owing to different motivations and perspectives. We refer the interested reader to the excellent surveys \cite{AFOV,HHHH,PV}. Using the geometric measure of entanglement (considered previously in \cite{Shimony, BL, WG}), Gross, Flammia and Eisert in \cite{GFE} argue that most quantum states are too entangled to be useful. On the other hand, observable states for any quantum computer based on photon interactions are symmetric quantum states. For symmetric tensors, the entanglement is much smaller as shown in \cite{FK}. They also show that most symmetric quantum states are close to being maximally entangled. In this paper, we generalize and improve some of the results in \cite{GFE, FK}\footnote{The results in \cite{GFE,FK} use an concentration of measure inequality that needs a correction. This is Equation~(2) in \cite{GFE} (and Lemma~3.7 in \cite{FK}). This does not disrupt their results significantly, but they do require a correction.}.

 We consider here two entanglement measures that are related to (non-euclidean) norms on $V$. Specifically we consider nuclear and spectral norms. These norms have applications to tensor completion and low rank tensor approximation, see for e.g., \cite{FO,NW,YZ}. In general, the computation of these norms is NP-hard, see \cite{FL,FL2}. In \cite{Nie}, computation of nuclear norms for symmetric tensors is addressed.
    
\begin{definition}\label{def1}
The {\em spectral norm} $\|T\|_\sigma$ of a tensor $T\in V$ is defined as the maximal value of $|\langle T,u\rangle|,$ where $u$ ranges over all pure tensors 
with $\|u\|=1$.
\end{definition}
\begin{definition}\label{def2}
The {\em nuclear norm} $\|T\|_\star$ of a tensor $T\in V$ is the minimal value of $\|u_1\|+\|u_2\|+\cdots+\|u_r\|$ over all decompositions
$T=u_1+u_2+\cdots+u_r$ where $r$ is a nonnegative integer and $u_1,u_2,\dots,u_r$ are pure tensors.
\end{definition}
The maximum in Definition~\ref{def1} exists because the unit sphere is compact, and one can show that the minimum in Definition~\ref{def2} exists as well.
%From  the Cauchy-Schwarz inequality follows that
%\begin{equation}\label{eq1}
%\|T\|_\sigma\leq \|T\|.
%\end{equation}
The spectral and nuclear norms are dual to each other (see \cite{DFLW}). In particular, we have
\begin{equation}\label{eq2}
|\langle T,S\rangle|\leq \|T\|_{\sigma}\|S\|_\star.
\end{equation}

The following measures of entanglement were considered in \cite{DFLW}.
\begin{definition} \label{geom.m.e}
For a unit tensor $T$ we define two entanglement measures by 
$$
E(T)=-2\log_2(\|T\|_\sigma)\mbox{ and }F(T)=2\log_2(\|T\|_\star).
$$
\end{definition}
From (\ref{eq2}), it follows that $\|T\|_\sigma\|T\|_\star\geq \langle T,T\rangle=1$ and  $F(T)\geq E(T)$. From the Cauchy-Schwarz inequality follows $\|T\|_\sigma\leq \|T\|=1$ and $E(T)\geq 0$.  It is also known that the maximum possible entanglement in both measures is the same, see \cite[Corollary~6.1]{DFLW}. The measure $E(T)$ is the geometric measure of entanglement that we mentioned before.

\subsection{General tensors}
We consider the setting when $V_1=V_2=\cdots=V_k=\C^n$, and $V=\C^n\otimes \C^n\otimes \cdots\otimes \C^n=(\C^n)^{\otimes k}$. The following result was obtained in \cite{QZ}, but we include a proof in Section~\ref{main} for completeness.
\begin{proposition} \label{upper bounds}
For all unit length tensors $T \in (\C^n)^{\otimes k}$, we have 
$$E(T)\leq F(T) \leq (k-1) \log_2(n).$$
\end{proposition}

We adapt a technique often found in coding theory to show the existence of tensors whose entanglement is very close to the upper bounds given above. 

\begin{theorem} \label{main}
Let $n,k \geq 2$, and let 
$$C = \min \{||T||_\sigma \ | \ T \in  (\C^n)^{\otimes k} \text{ with }  ||T|| = 1 \}.$$ 
Then for any $0 < \varepsilon < 1$, we have 
$$C^2 \leq \frac{ k(n-1) \ln(\frac{4}{\varepsilon})}{(n^k-1)(1-\varepsilon)^k}.$$
\end{theorem}

Let $E_{\max}$ denote the maximum entanglement of a unit length tensor in $(\C^n)^{\otimes k}$, i.e., 
let us define
$$
E_{\max} = \max\{E(T)\ |\ T \in (\C^n)^{\otimes k} \text{ with }  ||T|| = 1\}.
$$

\begin{corollary} \label{cor:main}
Let $n \geq 2$ be fixed. Then for $k \geq 21$, we have
 $$E_{\max} \geq (k-1)\log_2(n) - \log_2(k) - \log_2(\ln(k)) - 2.$$
\end{corollary}

%\begin{corollary} \label{cor:qubit}
%Let $n = 2$, and $k \geq 21$. Then, we have 
%$$E_{\max} \geq k - \log_2(k) - \log_2(\ln(k)) - 2.$$
%\end{corollary}

The result in \cite{GFE} is a concentration of measure result in the case of $n=2$. With a slight modification of the argument we can also recover such a result. Note that the set of all unit length tensors in $(\C^2)^{\otimes k}$ is just the (real) sphere $S^{2(2^k)-1} \subseteq (\C^2)^{\otimes k}$. In particular, the set of unit length tensors is a sphere which is compact, and has a standard measure on it.

\begin{corollary} \label{fraction}
Let $n = 2$, and $k \geq 4$. The fraction of unit length tensors $T$ in $(\C^n)^{\otimes k}$ such that 
%$$E(T) < k - \log_2(k) - \log_2(\ln(k) + \ln(16) + 1) - 1$$ 
$$
E(T) < k - \log_2(k) - \log_2(\ln(k)) - 3
$$
is at most $e^{-k}$. 
\end{corollary}

This improves significantly the bound of $k - 2\log_2(k) - 3$ in \cite[Theorem~2]{GFE}. Similar results can be shown for higher $n$ but we omit the details. So, most tensors have a high entanglement. Yet, finding an explicit\footnote{This can be given a formal definition, i.e., the usual one in complexity.} tensor with a high entanglement seems quite difficult --  this phenomenon is very familiar to researchers in computational complexity. For many complexity measures, one can show that most instances (or a random instance) has high complexity. However, constructing ``explicit" instances with high complexity is extremely challenging, and in many cases still unresolved. Example where explicit instances with high complexity are still far beyond current techniques include the celebrated P vs NP problem (and its algebraic analog VP vs VNP), tensor rank and Waring rank.

Loosely speaking, the geometric measure of entanglement is a measure of complexity. It is an interesting problem to construct explicit examples with a large entanglement.

\begin{problem}
Construct explicit tensors with large entanglement.
\end{problem}

We give a construction based on the determinants that constructs explicit tensors whose entanglement is quite large, but still falls short of the bounds in Corollary~\ref{cor:main} and Corollary~\ref{fraction}. We define the determinant tensor 
$$
{\rm det}_n = \sum_{\pi \in S_n} e_{\pi(1)} \otimes e_{\pi(2)} \otimes \dots \otimes e_{\pi(n)} \in (\C^n)^{\otimes n}.
$$

Consider the tensor $\det_{n^p} \in (\C^{n^p})^{n^p}$. Identifying $\C^{n^p}$ with $(\C^n)^{\otimes p}$, we can think of $\det_{n^p}$ as a tensor in $(\C^n)^{\otimes pn^p}.$ We define 
$$
T_{n,p} = \frac{1}{\sqrt{n^p !}} {\rm det}_{n^p} \in (\C^n)^{\otimes pn^p}.
$$

Since $||\det_{n^p}|| = \sqrt{{n^p}!}$, we have $||T_{n,p}|| = 1$. 

\begin{proposition} \label{explicit}
Let $n$ be fixed. For the unit length tensor $T_{p,n} \in (\C^n)^{\otimes k}$, with $k = pn^p$, we have 
$$
E(T_{n,p}) \geq k\log_2(n) - o(k).
$$
 
\end{proposition}

Let us observe that Theorem~\ref{main} is valid as long as $n,k \geq 2$. So, one can also consider the situation when $k$ is fixed. In this case, $\ln(\frac{4}{\varepsilon})/(1-\varepsilon)^k$ is simply a constant for any $0< \varepsilon < 1$. One can optimize the value of this constant by choosing an appropriate $\varepsilon$, but we omit that analysis. We obtain:

\begin{corollary} \label{conc.meas}
For fixed $k$, we have $$E_{\max} \geq (k-1)\log_2(n) + O(1).$$
\end{corollary}

\subsection{Symmetric tensors}
There is a natural action of $\Sigma_m$, the symmetric group on $m$ letters on $(\C^n)^{\otimes m}$ by permuting the tensor factors. A tensor $T \in (\C^n)^{\otimes m}$ is called a symmetric tensor if it is invariant under the action of $\Sigma_m$. We consider the space of symmetric tensors $\Sm^m(\C^n) \subseteq (\C^n)^{\otimes m}$. One can identify in a standard way $\Sm^m(\C^n)$ with polynomials of degree $m$ in $n$ variables. 

Symmetric tensors cannot have very large entanglement as was shown in \cite{FK}. We restrict our attention to symmetric tensors of norm $1$, which we denote by $\Sm_1^m (\C^n)$.

\begin{theorem} [\cite{FK}]
For all $T \in \Sm_1^m(\C^n)$, we have 
$$
E(T) \leq \log_2(d_{n,m}),
$$
where $d_{n,m} = {m+n-1 \choose m}$.
\end{theorem}

For fixed $n$, the existence of tensors $T \in  \Sm_1^m(\C^n)$ with $E(T) \geq \log_2(d_{n,m}) - \log_2(\log_2 d_{n,m}) - O(1)$ for $m$ sufficiently large is also shown in \cite{FK}. Our techniques can be extended to this case and yield improvements. 
Let $d_{n,m} = {m+n-1 \choose m}$. 

\begin{theorem} \label{theo:sym}
Let $C_s = \min\{||T||_\sigma \ |\ T \in \Sm_1^m(\C^n)\}$. Then for any $0 < \varepsilon < 1$, we have
$$
C^2_s \leq m^2 \varepsilon^2 + \frac{(n-1) \ln(4/\varepsilon)}{d_{n,m}-1}.
$$ 
\end{theorem}

Let us define 
$$
E^s_{\max} = \max\{ E(T)\ |\ T \in \Sm_1^m(\C^n)\}.
$$

\begin{corollary} \label{sym:main}
Let $n \geq 2$ be fixed. Then for $m$ sufficiently large, we have 
$$
E^s_{\max} \geq \log_2(d_{n,m}) - \log_2(\ln(d_{n,m})) - \log_2(n).
$$
\end{corollary}

For $n= 2$, we obtain a slightly stronger result.
\begin{corollary} \label{sym:qubit}
Let $n=2$. Then for all $m$, we have
$$
E^s_{\max} \geq \log_2(m) - \log_2(1 + \ln(4m\sqrt{m})).
$$
\end{corollary}

The strategy in \cite{FK} closely resembles that of \cite{GFE}. One significant difficulty they must overcome is to construct `$\epsilon$-nets'. We note that our methods bypass the need for this construction. Just like the results in \cite{GFE}, the results in \cite{FK} are stronger than just showing the existence of tensors with large entanglement. They show a (quantitative) concentration of measure result that says that most symmetric tensors are maximally entangled, see \cite[Theorem~1.2]{FK}. With our techniques, we can improve upon these results as well by pursuing a strategy similar to Corollary~\ref{fraction}, but we omit the details.

\section{Preliminaries}
Let $S^{n} \subset \R^{n+1}$ denote the $n$-sphere. Observe that the set of unit vectors in $\C^n = \R^{2n}$ is simply $S^{2n-1}$. 
So, the set of unit tensors in $(\C^n)^{\otimes k} = \C^{n^k}$ is just the sphere $S^{2n^k - 1}$. 

\begin{definition}
For $v \in S^{2n-1}$, i.e., $v$ of unit length in $\C^n$, we define 
$$B(v,\varepsilon) = \{w \in S^{2n-1}  | \ |\langle v,w \rangle |^2 \geq 1-\varepsilon\} \subseteq S^{2n-1}.$$ We call $B(v,\varepsilon)$ the $\varepsilon$-ball around $v$.
\end{definition}

\begin{remark}
Note that the $\varepsilon$-ball as defined above is not a ball in the usual metric. However, it is better suited for our purposes since the spectral norm is defined in terms of inner products. 
\end{remark}

\begin{proposition} \label{vol.b}
$$\vol(B(v,\varepsilon)) = \displaystyle \frac{ 2 \pi^n \varepsilon^{n-1}}{(n-1)!}.$$
\end{proposition}

Before proving the proposition, let us observe that $B(v,1)$ is the entire sphere $S^{2n-1}$. The volume of the sphere is usually calculated by a recursion. As a sanity check, when we plug in $\varepsilon = 1$ in the above formula, we do recover the (well known) volume of the sphere.

\begin{corollary} \label{vol.s} We have
$$\vol(S^{2n-1}) =  \vol(B(v,1)) =  \displaystyle \frac{2 \pi^n}{(n-1)!}$$
\end{corollary}

\begin{corollary} \label{frac}
We have $$\displaystyle \frac{ \vol(B(v,\varepsilon))}{\vol(S^{2n-1})}  = \varepsilon^{n-1}.$$
\end{corollary}

\begin{proof} [Proof of Proposition~\ref{vol.b}]
Without loss of generality, let us assume $v = (0,0,\dots,0,1) \in \C^n$. Then $B(v,\varepsilon)  = \{w\ | \ |w_n|^2 \geq  1 - \varepsilon \} \subseteq S^{2n-1}$. In real coordinates, this is the surface described by 
\begin{itemize}
\item[A.] $x_1^2 + y_1^2 + x_2^2 + \dots x_n^2 + y_n^2 = 1$ and 
\item[B.] $x_n^2 + y_n^2 \geq 1 - \varepsilon$ (or equivalently, $x_1^2 + y_1^2 + \dots x_{n-1}^2 + y_{n-1}^2 \leq \varepsilon$). 
\end{itemize}

If we restrict to $y_n \geq 0$, we get precisely half the surface, and we compute the volume of this half of the surface by writing it as a parametrized surface. We write 
$$
y_n = f(x_1,y_1,\dots,x_{n-1},y_{n-1},x_n) =  \sqrt{1 - x_n^2 - R^2} = \sqrt{1 - x_1^2 - y_1^2 - \dots - x_n^2},
$$
 where $R = \sqrt{x_1^2 + y_1^2 + \dots + x_{n-1}^2 + y_{n-1}^2}$. Now, the domain for the parametrized surface is $\{(x_1,y_1,\dots,x_n,y_n)\ |\ R^2 \leq \varepsilon, - \sqrt{1 - R^2} \leq x_{n} \leq \sqrt{1 - R^2}\}$. The volume of this surface is given by the following formula:

$$ \int_{D(0,\sqrt\varepsilon)} \int_{x_n = - \sqrt{1 - R^2}}^{\sqrt{1-R^2}} \left({\textstyle \sqrt{1 + \left(\frac{\partial f}{\partial x_1}\right)^2 +  \left(\frac{\partial f} {\partial y_1}\right)^2 + \dots + \left(\frac{\partial f}{\partial y_{n-1}}\right)^2 + \left(\frac{\partial f}{\partial x_n}\right)^2}} \right) dx_n dx_1dy_1 \dots dy_{n-1}, 
$$

where $D(0,\sqrt\varepsilon) = \{ (x_1,x_2,\dots,x_{n-1},y_{n-1})\ |\  x_1^2 + y_1^2 + \dots + y_{n-1}^2 \leq \varepsilon \}$. We compute $\left(\frac{\partial f} {\partial y_i}\right) = -  y_i/f$, and $\left(\frac{\partial f} {\partial x_i}\right) = -  x_i/f$, and so we get that the integrand is simply $1/f = 1/\sqrt{1 - x_n^2 - R^2}$. In other words, we have
$$ \int_{D(0,\sqrt\varepsilon)} \int_{x_n = - \sqrt{1 - R^2}}^{\sqrt{1-R^2}} \frac{1}{\sqrt{1 - R^2 - x_n^2}} dx_n dx_1dy_1 \dots dy_{n-1},$$

But now applying the formula $\int_{-a}^a \frac{1}{\sqrt{a^2 - x^2}} dx = \left[ {\rm sin}^{-1}(\frac{x}{a}) \right]_{x = -a}^{x = a} = \pi$, for $a = \sqrt{1-R^2}$, the integral simplifies to 

$$\int_{D(0,\sqrt\varepsilon)} \pi dx_1dy_1 \dots dy_{n-1} = \pi \cdot \vol(D(0,\sqrt\varepsilon)) = \pi \left(\frac{\pi^{n-1} \varepsilon^{n-1}}{(n-1)!} \right) = \frac{\pi^n\varepsilon^{n-1}}{(n-1)!}.$$

This concludes the computation for precisely half the surface, and so we multiply by two to get the required result.
\end{proof}

\begin{proposition} \label{in.pdt.calc}
Let $v,z,w \in S^{2n-1} \subset \C^n$ such that $|\langle v,z\rangle |^2 \geq 1- \varepsilon$ and $|\langle z,w\rangle |^2 \geq 1-\varepsilon$. Then $|\langle v,w\rangle |^2 \geq 1 - 4\varepsilon.$
\end{proposition}

\begin{proof}
By symmetry, we may assume $|\left< z,w \right>|^2 \geq |\left< v , z \right>|^2  = 1- \delta$ for some $\delta \leq \varepsilon$. Without loss of generality, let us assume $v = (0,\dots,0,1) \in \C^n$. Then we have $ \sqrt{1-\delta} =  |\langle v,z\rangle | = |z_n|$. For $x = (x_1,\dots,x_{n-1},x_n) \in \C^n$, define $\widetilde{x} = (x_1,\dots,x_{n-1}) \in \C^{n-1}$. Hence we have $z = (\widetilde{z},z_n)$ and $w = (\widetilde{w},w_n)$.

Now, we have 

\begin{align*}
\sqrt{1- \delta} & \leq | \langle z,w \rangle| \\
  & = | \langle \widetilde{z}, \widetilde{w} \rangle + z_n\overline{w_n}| \\
  & \leq | \langle \widetilde{z}, \widetilde{w} \rangle| + |z_n\overline{w_n}| \\
  & \leq ||\widetilde{z}|| \cdot || \widetilde{w} || + (\sqrt{1 - \delta}) |w_n| \\
  & = \sqrt{\delta} \sqrt{1 - |w_n|^2} + (\sqrt{1 - \delta}) |w_n|
\end{align*}

The last equality follows from the fact that for $x \in \C^n$, we have $||x ||^2 = ||\widetilde{x}||^2 + |x_n|^2$. In particular, we can rewrite the inequality as 
\begin{align*}
 \sqrt{1 - \delta}(1 - |w_n|) & \leq \sqrt{\delta} \sqrt{1 - |w_n|^2} 
\end{align*}

On squaring the terms, we get 
$$
(1 - \delta) (1 - |w_n|)^2 \leq \delta (1 - |w_n|^2).
$$

Now, dividing both sides by $1- |w_n|$, we get
$$
(1 - \delta) (1 - |w_n|) \leq \delta (1 + |w_n|).
$$

Rearranging the terms, we get $|w_n| \geq 1 - 2\delta$. We get the required conclusion since we have $|\langle v, w \rangle|^2 = |w_n|^2 \geq (1 - 2 \delta)^2 \geq 1 - 4 \delta \geq  1- 4 \varepsilon$.

\end{proof}

In the next proposition, we give an upper bound on the number of $\varepsilon$-balls needed to cover all of $S^{2n-1}$. In order to do this, we borrow a standard technique from coding theory and modify it appropriately.

\begin{proposition} \label{e-cover}
There exists $v_1,v_2,\dots,v_N \in S^{2n-1}$ such that $\bigcup\limits_{i=1}^N B(v_i, \varepsilon)$ is all of $S^{2n-1}$ for some $N \leq (\frac{4}{\varepsilon})^{n-1}$
\end{proposition}

\begin{proof}
 Let $B(v_1,\frac{\varepsilon}{4}),B(v_2,\frac{\varepsilon}{4}),\dots,B(v_N, \frac{\varepsilon}{4})$ be a maximal collection of non-interesecting $\frac{\varepsilon}{4}$-balls on $S^{2n-1}$. Observe that $N \leq \frac{\vol(S^{2n-1})}{\vol(B(v,\frac{\varepsilon}{4}))} = (\frac{4}{\varepsilon})^{n-1}$.

Now for any $w \in S^{2n-1}$, we have that $B(w,\frac{\varepsilon}{4}) \cap B(v_i,\frac{\varepsilon}{4})$ is non-empty for some $i$. Hence $\exists z \in S^{2n-1}$ such that $|\langle z,w \rangle|^2 \geq 1-\frac{\varepsilon}{4}$ and  $|\langle v_i,z \rangle|^2 \geq 1-\frac{\varepsilon}{4}$, and hence by Proposition~\ref{in.pdt.calc}, we have that $w \in B(v_i,\varepsilon)$.
Hence $\bigcup\limits_i B(v_i,\varepsilon)$ is all of $S^{2n-1}$. 
\end{proof}

\section{Main results} \label{sec:main}

\subsection{Upper bounds on entanglement}
We will derive the upper bounds in Proposition~\ref{upper bounds} by a simple induction argument.

\begin{proof} [Proof of Proposition~\ref{upper bounds}]
We prove this by induction on $k$. The base case $k = 1$ is trivial. 

Now, assume $k \geq 2$. We define

\begin{align*}
C(n,k) &:= \min\{ ||T||_\sigma \ | \ T \in S^{2n^k-1} \subset (\C^n)^{\otimes k} \} \\
&\  = \min \left\{ \frac{||T||_\sigma}{||T||} \ \bigg{|} \  T \in (\C^n)^{\otimes k} \right\}.
\end{align*}

Let $e_1,\dots,e_n$ denote the standard basis for $\C^n$. We can write any $T \in (\C^n)^{\otimes k}$ as $T = \sum_{i=1}^n T_i \otimes e_i$. Further, observe that $||T||_\sigma \geq ||T_i||_\sigma$, since for any unit length tensor $u \in (\C^n)^{\otimes k-1}$, we have $|\langle T_i,u \rangle| \leq |\langle T, u \otimes e_i\rangle|$.

Hence, we have 

\begin{align*}
||T||^2 & = \sum_{i=1}^n ||T_i||^2 \\
   & \leq \sum_{i=1}^n \frac{||T_i||_\sigma^2}{C(n,k-1)^2}.\\
   & \leq \sum_{i=1}^n \frac{||T||_\sigma^2}{C(n,k-1)^2}. \\
   & = n \frac{||T||_\sigma^2}{C(n,k-1)^2}
   \end{align*}

In particular for unit length tensors $T$, we have $||T||_\sigma \geq C(n,k-1)^2/n$, and hence 
\begin{align*}
E(T) &= - 2 \log_2(||T||_\sigma) \\ 
&\leq - 2 \log_2(C(n,k-1)) + \log_2(n) \\
& \leq (k-2)\log_2(n)  + \log_2(n)\\
& = (k-1) \log_2(n),
\end{align*}
 where the second inequality follows by induction. 
 
 Recall that maximum possible entanglement in both measures $E(T)$ and $F(T)$ is the same (\cite[Corollary~6.1]{DFLW}). So, we get the result for $F(T)$ as well.

\end{proof}

\subsection{Lower bounds on entanglement}
In this section, we will give proofs of our main lower bound results on entanglement for general tensors. For this section, we will fix $v_1,v_2,\dots,v_N \in S^{2n-1} \subset \C^n$ such that $\bigcup\limits_{i=1}^N B(v_i, \varepsilon)$ is all of $S^{2n-1}$. The following proposition is the key technical result to derive Theorem~\ref{main}.
 
\begin{proposition}
Let $C = \min\{||T||_\sigma \ |\  T \in S^{2n^k-1}\subset (\C^n)^{\otimes k} \}$. Then, for all $T \in S^{2n^k -1}$, we have $|\langle T, v_{i_1} \otimes v_{i_2} \otimes \dots \otimes v_{i_k}\rangle|^2 \geq C^2(1-\varepsilon)^{k}$ for some $1 \leq i_1,i_2,\dots,i_k \leq N$. 
\end{proposition}

\begin{proof}
By definition, for any $T \in (\C^n)^{\otimes k}$, we have $|\langle T,a_1 \otimes a_2 \otimes \dots \otimes a_k \rangle| \geq C$ for some pure tensor $a_1 \otimes a_2 \otimes \dots \otimes a_k$.

The map $v \mapsto \langle T, v \otimes a_2 \otimes \dots \otimes a_k \rangle$ defines a element of $(\C^n)^*$, which we denote by $f$. Using the linear isomorphism $\C^n$ with $(\C^n)^*$ given by the inner product, we get a vector $v_f$ such that $f(w) = \langle v_f,w\rangle$. Since $\bigcup\limits_{i=1}^N B(v_i, \varepsilon)$ is all of $S^{2n-1}$, there exists $v_{i_1}$ such that $$|\langle v_f,v_{i_1} \rangle| \geq (\sqrt{1-\varepsilon}) ||v_f|| \geq (\sqrt{1-\varepsilon}) |\langle v_f,a_1 \rangle| \geq C\sqrt{1-\varepsilon}.$$

Hence $$|\langle T, v_{i_1} \otimes a_2 \otimes \dots \otimes a_k \rangle| = |\langle v_f,v_{i_1} \rangle| \geq C\sqrt{1-\varepsilon}.$$

Repeating the argument for all the tensor factors, we get a tensor $v_{i_1} \otimes v_{i_2} \otimes \dots \otimes v_{i_k}$ such that $$|\langle T, v_{i_1} \otimes v_{i_2} \otimes \dots \otimes v_{i_k}\rangle| \geq C(\sqrt{1-\varepsilon})^{k}.$$ 
\end{proof}

We write $[N] := \{1,2,\dots,N\}$ in the following corollary.

\begin{corollary}
We have $\bigcup\limits_{ i_1,i_2,\dots,i_k \in [N]^k} B(v_{i_1} \otimes v_{i_2} \otimes \dots \otimes v_{i_k}, 1-C^2(1-\varepsilon)^k)$ is all of $S^{2n^k -1}$.
\end{corollary}

Using Corollary~\ref{frac}, we get

\begin{corollary} \label{asdf}
 We have
$$ N^k \left((1-C^2(1-\varepsilon)^k)^{n^k-1} \right) \geq 1.$$
\end{corollary}

We have all the tools required to prove Theorem~\ref{main}.

\begin{proof} [Proof of Theorem~\ref{main}]
We can use the above inequality to get an upper bound for $C$. We have
\begin{align*}
1 &\leq  N^k (1-C^2(1-\varepsilon)^k)^{n^k-1} \\
 & \leq \left(\frac{4}{\varepsilon}\right)^{k(n-1)}(1-C^2(1-\varepsilon)^k)^{n^k-1}.
\end{align*} 
Taking logarithms we get
\begin{align*}
 0 & \leq k(n-1) \ln\left(\frac{4}{\varepsilon}\right) + (n^k-1) ln(1-C^2(1-\varepsilon)^k)\\
  &\leq k(n-1) \ln\left(\frac{4}{\varepsilon}\right) + (n^k-1) (-C^2(1-\varepsilon)^k).
\end{align*}\
This gives us
$$
C^2 \leq \frac{ k(n-1) \ln(\frac{4}{\varepsilon})}{(n^k-1)(1-\varepsilon)^k},
$$
as required.

\end{proof}

With the proof of Theorem~\ref{main}, we now proceed to deduce Corollary~\ref{cor:main} and Corollary~\ref{fraction}. These are a bit computational, and we will try to keep the computations transparent and succinct.

\begin{proof} [Proof of Corollary~\ref{cor:main}]
From Proposition~\ref{main}, we get 
$$
E_{\max} \geq \log_2\left(\frac{n^k -1}{n-1}\right) - \log_2(k) + \log_2(1-\varepsilon)^k - \log_2\left(ln\left(\frac{4}{\varepsilon}\right)\right).
$$
by taking a logarithm on both sides. 

First, observe that  $\log_2\left(\displaystyle\frac{n^k -1}{n-1}\right) \geq \log_2(n^{k-1}) = (k-1) \log_2(n)$. 

Next, set $\varepsilon = \delta/k$. So, $(1-\varepsilon)^k \geq 1 - k \varepsilon = 1 - \delta$. Also, note that $\log_2\left(ln\left(\frac{4}{\varepsilon}\right)\right) = \log_2 \left( \ln(k) + \ln(4/\delta) \right)$. Now, set $\delta = 4/e^3$. Then 
$$
\log_2(1-\varepsilon)^k \geq \log_2(1-\delta) = \log_2(1-4/e^3) \geq -1,
$$ and 
$$
\log_2 \left( \ln(k) + \ln(4/\delta) \right) = \log_2(\ln(k) + 3) \leq \log_2(2\ln(k)) = \log_2(\ln(k)) + 1.
$$
The inequality above follows because we assume $k \geq 21$ and consequently $\ln(k) \geq 3$.
Thus, we get
$$
E_{\max} \geq (k-1) \log_2 (n) - \log_2(k) - \log_2(\ln(k)) - 2
$$

%It is easy to see that $\ln(1-x) \geq -2x$ for $0 \leq x \leq 1/2$. Indeed, consider $f(x) = \ln(1-x) + 2x$ and note that $f(0) = 0$, and $f'(x) = 2 - 1/1-x \geq 0$ for $x \leq 1/2$. Consequently, $1-x \geq e^{-2x}$ for $0 \leq x \leq 1/2$. Also, note that $(1-\varepsilon)^k \geq 1 - k \varepsilon$ (since $\varepsilon < 1$).

%Now, set $\varepsilon = \delta/k$ and assume $\delta \leq \frac{1}{2}$. Then by the above observations, we get that $\log_2(1-\varepsilon)^k = \log_2(1 - \delta/k)^k  \geq \log_2(1-\delta) \geq \log_2(e^{-2\delta}) = -2\delta \log_2(e) \geq -3\delta$. Thus, we get  
%$$
%E_{\max} \geq (k-1) \log_2 (n) - \log_2(k) -3 \delta - \log_2\left(\ln(k) + \ln\left(\frac{4}{\delta}\right)  \right).
%$$

%Further, suppose $\delta = 4/e^3$. Then $3 \delta + \log_2\left(\ln(k) + \ln\left(\frac{4}{\delta}\right)  \right) = 12/e^3 + \log_2(\ln(k)  + 3) \leq \log_2(\ln(k)) + 2$. This follows from the fact that $12/e^3 \leq 1$, and that for $k \geq 21$, we have $\log_2(\ln(k) + 3) \leq \log_2(2\ln(k)) \leq \log_2(\ln(k)) + 1.$ This gives 
%$$
%E_{\max} \geq (k-1) \log_2 (n) - \log_2(k) - \log_2(\ln(k)) - 2
%$$
%as desired.

\end{proof}

%\begin{proof} [Proof of Corollary~\ref{cor:qubit}]
%This proof is similar to the one above. We have
%$$
%E_{\max} \geq \log_2\left(2^k -1\right) - \log_2(k) + \log_2(1-\varepsilon)^k - \log_2\left(ln\left(\frac{4}{\varepsilon}\right)\right).
%$$
%Observe from the previous proof that for $k \geq 21$, we get $\log_2(1-\varepsilon)^k - \log_2\left(ln\left(\frac{4}{\varepsilon}\right)\right) \geq -\log_2(\ln(k))- 1 - 8/e^3 \geq \log_2(\ln(k)) -1.40 $ when we set $\epsilon = \delta/k$, with $\delta = 4/e^3$. Further, $\log_2(2^k-1) =  k + \log_2(1-2^{-k}) \geq k - 0.5$ follows from $k \geq 21$. This gives the required result.
%\end{proof}

We can now prove Corollary~\ref{fraction}

\begin{proof} [Proof of Corollary~\ref{fraction}]
Let $D$ be the infimum of all numbers $d$ such that the fraction of unit length tensors $T$ having $||T||_\sigma > d$ is at most $e^{-k}$. We want a lower bound on $-2\log_2(D)$. Similar to Corollary~\ref{asdf}, we can deduce
$$ 
N^k \left((1-D^2(1-\varepsilon)^k)^{n^k-1} \right) \geq e^{-k}.
$$ 

Again, using that $N \leq \left(\frac{4}{\varepsilon}\right)$, and following the steps in the proof of Theorem~\ref{main}, we get 

$$
D^2 \leq \frac{k \ln(\frac{4}{\varepsilon}) + k}{ (2^k -1)(1-\varepsilon)^k }
$$

Taking natural logarithms on both sides, we get
$$
-2 \log_2 (D) \geq \log_2(2^k-1) + \log_2 (1 - \varepsilon)^k  - \log(k) - \log_2(\ln(4/\varepsilon) + 1).
$$

First, note that for $k \geq 4$, we have $\log_2(2^k-1) \geq k - \frac{1}{4}$. Now, set $\epsilon = 1/4k$. Then we get $\log_2(1 - \varepsilon)^k \geq \log_2(1-k\varepsilon) = \log_2(3/4) > -3/4$. We also get $
\log_2(\ln(4/\varepsilon) + 1) = \log_2(\ln(16k) + 1).$

Now, for $k \geq 4$, one can check that $\ln(16k) + 1 = \ln(k) + \ln(16) + 1 \leq 4 \ln(k)$. Thus $\log_2(\ln(16k) + 1) \leq \log_2(\ln(k)) + 2$. Thus, we have:
$$
-2 \log_2 (D) \geq (k  - 1/4) - \log(k) - (3/4) -  \log_2(\ln(k)) - 2.
$$
Thus, we get
$$
-2 \log_2 (D) \geq k  - \log(k) - \log_2(\ln(k)) - 3
$$
as required.
\end{proof}

\subsection{An explicit tensor with high entanglement}
In this section, we prove Proposition~\ref{explicit}, i.e., we give a lower bound on the entanglement of an explicit tensor ($T_{n,p}$ as defined in Proposition~\ref{explicit}). Recall that 
$$
{\rm det}_n = \sum_{\pi \in S_n} e_{\pi(1)} \otimes e_{\pi(2)} \otimes \dots \otimes e_{\pi(n)} \in (\C^n)^{\otimes n}.
$$
We know that $||\det_n||_\sigma = 1$, by \cite{Derksen}. Consider the tensor $\det_{n^p} \in (\C^{n^p})^{\otimes n^p}$. By identifying $\C^{n^p}$ with $(\C^n)^{\otimes p}$, we identify $(\C^{n^p})^{\otimes n^p}$ with $(\C^n)^{\otimes pn^p}$. Thus we can consider $\det_{n^p}$ as a tensor in $(\C^n)^{\otimes pn^p}$. Note that considering $\det_{n^p}$ as a tensor in $(\C^n)^{\otimes pn^p}$ cannot increase the spectral norm since the set of pure tensors of unit length in $(\C^n)^{\otimes pn^p}$ is a subset of pure tensors in $(\C^{n^p})^{\otimes n^p}$. 

On the other hand, if $e_1,e_2,\dots,e_n$ is a standard basis for $\C^n$, we define $e_I = e_{i_1} \otimes \dots \otimes e_{i_p}$ for $I = (i_1,\dots,i_p) \in [n]^p$ where $[n] = \{1,2,\dots,n\}$. Fix any bijection $h:[n^p] \rightarrow [n]^p$. We define
$$
u = \bigotimes\limits_{i=1}^{n^p} e_{h(i)}.
$$

Clearly, $u$ is a pure tensor in $(\C^n)^{\otimes pn^p}$, and $|\langle \det_{n^p}, u \rangle| = 1$. Hence, $||\det_{n^p}||_\sigma = 1$, when considered as a tensor in $(\C^n)^{\otimes pn^p}.$ Now, consider the unit length tensor $T_{n,p}$. 

\begin{proof} [Proof of Proposition~\ref{explicit}]
By the above discussion, we have $||T||_\sigma = \frac{1}{\sqrt{n^p !}}$. Hence, we get $E(T) = 2 log_2 \sqrt{n^p !} = log_2( n^p!).$ Using stirling's formula, we compute
\begin{align*}
log_2(n^p !) & = n^p \log_2(n^p) - \log_2 (e)(n^p) + O(\log_2(n)) \\
& = pn^p \log_2(n) - o(pn^p).
\end{align*}

The proposition follows by replacing $pn^p$ with $k$. 

\end{proof}

\section{Symmetric tensors}
This section is devoted to the study of entanglement for symmetric tensors. The overall strategy parallels the case of general tensors. However, there is one additional ingredient that we will need, and that is Banach's theorem below.

Analogous to the case of general tensors, we define 
$$
C_s := \min\{||T||_\sigma^2 \ |\ T \in \Sm_1^m(\C^n)\}.
$$
Our goal is lower bounds on $C_s$. The following is the aforementioned result of Banach (see \cite{Banach}, \cite[Theorem~2.7]{FK}).

\begin{theorem} [Banach] \label{theo:banach}
For $T \in \Sm^m(\C^n)$, we have $||T||_\sigma = \max\{ |\left< T, v^{\otimes m} \right>| \ : \ ||v|| = 1\}$.
\end{theorem}

Observe that by Proposition~\ref{e-cover}, we can take $v_1,\dots,v_N$, with $N \leq (4/\varepsilon)^{n-1}$ such that $\cup_{\alpha =1}^N B(v_\alpha,\varepsilon)$ is all of $S^{n-1}$. For this entire section $v_1,\dots,v_N$ will denote such a choice. Using Banach's theorem above, we deduce the main technical lemma for the purposes of this section.

\begin{lemma}
For all $T \in \Sm_1^m(\C^n)$, we have $| \left< T ,v_I^{\otimes m} \right>|^2 \geq C^2 - m^2\varepsilon^2$ for some $I \in \{1,2,\dots,N\}$.
\end{lemma}

\begin{proof}
It follows from Theorem~\ref{theo:banach} that for any $T \in \Sm_1^m(\C^n)$, $\exists v \in \C^n$ with $||v|| =1$ such that $|\left< T, v^{\otimes m} \right>| \geq C_s$. By the choice of $v_1,\dots,v_N$, we have $|\left<v,v_I\right>|^2 \geq 1 - \varepsilon$ for some $I \in \{1,2,\dots,N\}$.

Let $e_1,\dots,e_n$ denote an orthonormal basis for $\C^n$. Without loss of generality, we can assume $v_I = e_n$. The vector space $(\C^n)^{\otimes m}$ has an orthonormal basis $\{e_{i_1i_2\dots i_m} := e_{i_1} \otimes \dots \otimes e_{i_m}\ |\ 1 \leq i_1,i_2,\dots,i_m  \leq n\}$. We can write 
$$
(\C^n)^{\otimes m} = \C e_{nn\dots n} \oplus V,
$$
where $V = \bigoplus\limits_{i_1\dots i_m \neq nn\dots n} \C e_{i_1\dots i_m}$

Given $T \in (\C^n)^{\otimes m}$, we can write $T = T_{nn\dots n} + \widetilde{T}$ uniquely where $T_{nn\dots n} \in \C e_{nn\dots n}$ and $\widetilde{T} \in V$.

Now, let $v = (v_1,\dots,v_n)$ be the coordinates of $v$ in the basis $e_1,\dots,e_n$. Then, we have

$$ 1 - \varepsilon \leq |\left<v,v_I\right>|^2 = |\left<v,e_n\right>|^2 = |v_n|^2. $$

Thus, we have $|(v^{\otimes m})_{nn \dots n}|^2 = |v_n^m|^2 = (|v_n|^2)^m \geq (1 - \varepsilon)^m \geq 1 - m\epsilon$. In particular, this implies that $||\widetilde{v^{\otimes m}} || \leq m\varepsilon$.

Now, we have
\begin{align*}
C_s^2 \leq |\left<T,v^{\otimes m} \right>|^2 &\leq | \left<\widetilde{T},\widetilde{v^{\otimes m}} \right>  + T_{nn \dots n} \overline{(v^{\otimes m})_{nn\dots n}}|^2 \\
& \leq |\left<T, \widetilde{v^{\otimes m}} \right> |^2 + |T_{nn\dots n}|^2|(v^{\otimes m})_{nn\dots n}|^2 \\
& \leq m^2\varepsilon^2 + | \left<T, v_I^{\otimes m} \right>|^2 |(v^{\otimes m})_{nn \dots n}|^2 \\
& \leq m^2 \varepsilon^2 +| \left<T, v_I^{\otimes m} \right>|^2. 
\end{align*}

The first and second inequalities are clear. The third inequality follows from $||T|| = 1$, $||\widetilde{v^{\otimes m}} || \leq m\varepsilon.$ and the fact that $T_{nn\dots n} = \left<T,e_n^{\otimes m}\right> = \left<T,v_I^{\otimes m} \right>$ (since $v_I = e_n$). The last inequality follows from $|v^{\otimes m}_{nn\dots n}|^2 = |v_n|^{2m} \leq 1$.

Hence, we get 

$$
|\left<T, v_I^{\otimes m} \right>|^2 \geq C_s^2 - m^2 \varepsilon^2.
$$
\end{proof}

Friedland and Kemp show that we have an isometry between $\Sm_1^m(\C^n)$ and $S^{2d_{n,m} -1}$. Along with the above lemma, we obtain:

\begin{corollary}
We have $\cup_{\alpha =1}^N B(v_\alpha^{\otimes m}, 1 - (C_s^2 - m^2\varepsilon^2))$ covers $\Sm_1^m(\C^n) = S^{2d_{n,m} - 1}$.
\end{corollary}

Using Corollary~\ref{frac}, one obtains the following inequality.

\begin{corollary}
We have $(\frac{4}{\varepsilon})^{n-1} (1 - (C_s^2 - m^2 \varepsilon^2))^{d_{n,m}-1} \geq 1$.
\end{corollary}

\begin{proof} [Proof of Theorem~\ref{theo:sym}]
Applying the natural logarithm to the corollary above, we get
\begin{align*}
0 & \leq (n-1) \ln(4/\varepsilon) + (d_{n,m} -1) \ln (1 - (C_s^2 - m^2 \varepsilon^2)) \\
& \leq (n-1) \ln(4/\varepsilon) + (d_{n,m} -1) (- (C_s^2 - m^2 \varepsilon^2))
\end{align*}

Hence, we get 
$$
(d_{n,m} -1)  (C_s^2 - m^2 \varepsilon^2) \leq (n-1) \ln(4/\varepsilon).
$$

Equivalently, we get
$$
C_s^2 \leq m^2 \varepsilon^2 +  \frac{(n-1) \ln(4/\varepsilon)}{d_{n,m}-1}
$$
as required.
\end{proof}

Once again, by taking a natural logarithm on both sides in Theorem~\ref{theo:sym}, we obtain:
$$
-2 \log_2 C_s \geq -\log_2 \left(m^2 \varepsilon^2 +  \frac{(n-1) \ln(4/\varepsilon)}{d_{n,m}-1}\right)
$$

Replacing $\varepsilon$ by $\delta/m$, we get

\begin{equation} \label{dumb}
E^s_{\max} = -2 \log_2 C_s \geq -\log_2\left(\delta^2 +  \frac{(n-1) \ln(4m/\delta)}{d_{n,m}-1}\right)
\end{equation}

\begin{proof} [Proof of Corollary~\ref{sym:main}]
Substituting $\delta = (d_{n,m}-1)^{-1/2}$ in the above expression, we get
$$
E^s_{\max} \geq \log_2\frac{d_{n,m}-1} {1 +  (n-1) \ln(4m(d_{n,m}-1)^{1/2})} 
$$
For $m$ sufficiently large, we claim that
\begin{equation} \label{computation}
\frac{d_{n,m}-1} {1 +  (n-1) \ln(4m(d_{n,m}-1)^{1/2})} \geq \frac{d_{n,m}} {n \ln (d_{n,m})}
\end{equation}

from which we can conclude that
$$
E^s_{\max} \geq \log_2(d_{n,m}) -\log_2(\ln(d_{n,m})) - \log_2(n).
$$

So, it remains to prove (\ref{computation}) above. For $m \gg 0$, the following is easy to verify because $d_{n,m} = O(m^{n-1})$.
$$
1 + (n-1)\ln(4m(d_{n,m}-1)^{1/2}) \leq 1 + (n-1)\ln(4m(d_{n,m})^{1/2}) \leq 0.99 n\ln(d_{n,m}).
$$
This means that 
$$
\frac{1 + (n-1)\ln(4m(d_{n,m}-1)^{1/2})}{n \ln(d_{n,m})} \leq 0.99 \leq 1 - \frac{1}{d_{n,m}} = \frac{d_{n,m}-1}{d_{n,m}},
$$
which rearranged gives us (\ref{computation}) as required.
\end{proof}

\begin{proof} [Proof of Corollary~\ref{sym:qubit}]
Setting $n = 2$, (\ref{dumb}) simplifies to
$$
E^s_{\max} \geq -\log_2\left(\delta^2 + \frac {\ln(4m/\delta)}{m}\right) = \log_2(m) - \log_2(m \delta^2 + \ln(4m/\delta)).
$$

If we take $\delta = \frac{1}{\sqrt{m}}$, we get 
$$
E^s_{\max} \geq \log_2(m) - \log_2(1 + \ln(4m \sqrt{m})),
$$
\end{proof}

\end{document}